\documentclass[12pt,a4paper]{amsart}
\usepackage{mathtools}
\usepackage{amssymb}
\usepackage{amsthm}
\usepackage{enumitem}
\usepackage{amsmath}
\usepackage{url}
\usepackage{graphicx} 
\usepackage{subcaption}
\usepackage{hyperref}
\usepackage{comment}
\usepackage{cite}
\usepackage{ytableau}
\usepackage{algpseudocode}
\usepackage{algorithm}
\DeclarePairedDelimiter{\gen}{\langle}{\rangle}

\newtheorem{theorem}{Theorem}
\newtheorem{corollary}[theorem]{Corollary}
\newtheorem{lemma}[theorem]{Lemma}
\theoremstyle{definition}

\theoremstyle{remark}

\theoremstyle{definition}
\newtheorem{proposition}[theorem]{Proposition}
\theoremstyle{definition}

\theoremstyle{definition}

\theoremstyle{definition}

\usepackage{xcolor}

\title{The connectivity of the normalising and permuting graph of a finite soluble group}
\author{Eoghan Farrell}
\author{Chris Parker}
\address{Eoghan Farrell \\
School of Mathematics\\
University of Birmingham\\
Edgbaston\\
Birmingham B15 2TT\\
United Kingdom}\email{exf051@student.bham.ac.uk}
\address{Chris Parker\\
School of Mathematics\\
University of Birmingham\\
Edgbaston\\
Birmingham B15 2TT\\
United Kingdom}  \email{c.w.parker@bham.ac.uk}

\begin{document}
\maketitle

\maketitle \pagestyle{myheadings}

\markright{{\sc Normalising and permuting graphs }}   \markleft{{\sc Eoghan Farrell and Chris Parker}}

\begin{abstract}
We introduce the normalising graph of a group and study the connectivity of the normalising and permuting graphs of a group when the group is finite and soluble. In particular, we classify finite soluble groups with disconnected normalising graph. The main results shows that if a finite soluble group has connected normalising graph then this graph has diameter at most \(6\). Furthermore, this bound is tight.  A corollary then presents the connectivity properties of the permuting graph.
\end{abstract}

\section{Introduction}

In recent years, there has been a spike of interest in problems related to graphs which are defined on groups. For a group \(G\), these graphs have vertex set consisting of some subset of \(G\) and two vertices form an edge if they satisfy a particular group theoretic property. In \cite{MR4346241}, Cameron describes a myriad of ways in which one can define such a graph. Cameron also portrays a hierarchy for the collection of graphs defined on a group. Here we investigate two such graphs: the normalising graph and the permuting graph.  These graphs are not included in the hierarchy given in \cite{MR4346241}. The \emph{normalising graph} of \(G\), $\Gamma(G)$,   has vertex set consisting of the non-trivial elements of $G$ and there is an edge between two distinct vertices if and only if one of the cyclic subgroups they generate normalises the other.  The \emph{permuting graph} of \(G\), $\Psi(G)$, also has vertex consisting of non-trivial elements of  \(G\) and there is an edge between two vertices in this graph if the cyclic subgroups they generate permute, or equivalently, their product is a subgroup. We recall that the commuting graph of $G$, $\mathrm K(G)$, also has vertex set the non-trivial elements of $G$, and two vertices are adjacent if and only if they commute. Thus the edge set of $\mathrm K(G)$ is contained in the edge set of $\Gamma(G)$ which itself is contained in the edge set of $\Psi(G)$. We mention here the \emph{soluble graph} $\Sigma(G)$ in which $x$ and $y$ form an edge if and only if $\langle x,y\rangle$ is soluble (see \cite{MR4513813}).  By Ito's theorem \cite{MR71426}, if $x$ and $y$ are joined in $\Psi(G)$, then $\langle x,y\rangle$ is metacyclic and so soluble.  Hence the edge set of $\Psi(G)$ is contained in the edge set of $\Sigma(G)$. Of course, we could also define the \emph{metacyclic graph}. In this way, we obtain an alternative hierarchy for a collection of graphs defined on groups.   The permuting graph and the normalising graph have not been studied as extensively as the commuting graph.
 In \cite{MR4513813} it was shown that the diameter of the soluble graph of a finite group is at most \(5\).
  Following \cite{MR2371966}, we define the \emph{Engel graph}, $\mathcal E(G)$, to have vertices the non-trivial elements of \(G\) and we join two vertices \( x\) and \(y\) if and only if \([\langle y\rangle, \langle x\rangle;n]=1\) or \([ \langle x\rangle,\langle y\rangle ;n]=1\) for some natural number $n$. The connectivity of $\mathcal E(G)$ is studied in \cite{MR4529423}  and   \cite[Theorem 1.3 ]{MR4529423} states  that if \(G\) is soluble and \(G/Z_{\infty}(G)\) is not Frobenius, then   \(\mathcal E(G)\) has diameter at most \(4\). We remark that the normalising graph is a subgraph of  $\mathcal E(G)$. In \cite{MR3081551},   Parker showed that the diameter of the commuting graph of a soluble group with trivial centre is at most \(8\). To the best of our knowledge, the connectivity of the normalising graph  has not been investigated.
 Our main theorem is

\begin{theorem}\label{Diam 6}
    Suppose that \(G\) is a finite soluble group and that  $\Gamma $ is its normalising graph.
    \begin{enumerate} \item[(i)] The graph $\Gamma $ is disconnected if and only if $G=KC$  is a   Frobenius group with kernel $K$ and complement $C$ such that for all primes  $p$ dividing $|C|$ and all primes $r$ dividing $|K|$, we have  \(p $ does not divide \(r-1\). Furthermore, the connected components of $\Gamma $ have diameter at most $2$.
    \item [(ii)]
  If $\Gamma $ is connected, then the diameter of $\Gamma$ is at most $6$.
    \end{enumerate}
    Furthermore, there exist finite soluble groups with normalising graph of diameter $6$.
\end{theorem}

   In \cite{MR3493269},   Rajkumar and  Devi   study forbidden subgraphs in the permuting graph. They then pose the problem of determining the connectivity and bounding the diameter of the permuting graph \cite[Problem 4.1]{MR3493269}.  Here we resolve this problem  for soluble groups.

\begin{corollary}\label{cor:permuting graphs} Suppose that $G$ is a finite soluble group. Then the permuting graph $\Psi(G)$ is connected if and only if the normalising graph $\Gamma(G)$ is connected. Furthermore, if $\Psi(G)$ is connected, then it has diameter at most $6$ and this bound is attained.\end{corollary}

\begin{proof} If $\Psi(G)$ is disconnected, then so is $\Gamma(G)$.
Assume that $\Gamma(G)$ is disconnected and that $\Psi(G)$ is connected.  Then $G$ is a Frobenius group by Theorem~\ref{Diam 6}. Write $G=KC$ where $K$ is the Frobenius kernel and $C$  a complement.  By
Lemma~\ref{connected components}, the connected components of $\Gamma(G)$ are $\Gamma(K)$ and $\Gamma(C^k)$, $k \in K$.  Hence, as $\Psi(G)$ is connected there exist non-trivial cyclic subgroups $A=\langle a\rangle$ of $C$ (or a conjugate) and  $B=\langle b \rangle$ of $K$ such that $AB$ is a subgroup. Since $AB \cap K=(A\cap K)B =B$ and $K$ is normal in $G$, we have $A$ normalises $B$. Hence $a$ and $b$ are joined in $\Gamma(G)$, which is a contradiction. Hence $\Psi(G)$ is disconnected.  This proves that $\Gamma(G)$ is connected if and only if $\Psi(G)$ is connected.   Finally,  the example in this article's final section exhibits a group $G$ with $\Psi(G)$ of diameter $6$.
\end{proof}

We follow the group theoretic notation introduced in \cite{MR0569209}. Throughout this article, we only consider finite soluble groups. If \(G\) is a group, then \(G^{\#}\) is the set of non-trivial elements in \(G\). If $X \subseteq G$, then $X^\# = X \cap G^\#$.  The set of primes which divide the order of \(G\) is denoted by \(\pi(G)\). For   \(x\in G\), the order of \(x\) is given by \(o(x)\).
The Fitting group of \(G\), denoted by \(F(G)\), is defined to be the maximal normal nilpotent subgroup of \(G\). If \(p\) is a prime, \( O_{p}(G)\) is the maximal normal \(p\)-subgroup of \(G\). If \(G\) is a \(p\)-group, \(\Omega_{1}(G)\) is the subgroup generated by all elements in \(G\) whose order divides \(p\).
Recall that a graph, \(\Gamma\), is connected if there is a path between any pair of distinct vertices in \(V(\Gamma)\). For distinct \(x,y\in V(\Gamma)\), we write \(x\sim y\) to denote the fact that \(x\) and \(y\) share an edge in \(\Gamma
\).  The distance function on  \(\Gamma\),  \(\mathrm{d}(x,y)\), is the length of the shortest path between the vertices \(x\) and \(y\) if such a path exists and otherwise   \(\mathrm{d}(x,y)=\infty\). The diameter of \(\Gamma\), denoted by \(\mathrm {diam}(\Gamma)\), is given by \(\mathrm {diam}(\Gamma)=\text{max} \{ \mathrm{d}(x,y)\mid x,y\in V(\Gamma)\}\). For \(H\subseteq V( \Gamma)\) and $x\in V(\Gamma)$, define \(\mathrm{d}(x,H)\) to be the length of a shortest path from \(x\) to a vertex of  \(H\). Equivalently, \(\mathrm{d}(x,H)=\min\{\mathrm{d}(x,h)\mid h \in H\}\).  We write \(x\sim H\) if $\mathrm{d}(x,H)=1$.

\renewcommand{\abstractname}{Acknowledgements}
\begin{abstract}
 The first author was supported by an EPSRC Undergraduate Vacation Internship which was administered by the University of Birmingham. Both authors thank the referee for a constructive report which led to an improvement of the readability of the paper and increased the clarity and precision of the proofs which lead to the main results.
\end{abstract}

\section{Preliminaries}
Before beginning the proof of Theorem \ref{Diam 6}, we introduce some group theoretical results that we use throughout.
 We begin with the following lemma about subgroups of \(p\)-groups.
\begin{lemma}\label{no non cyclic subgroup in P group}
    Let \(p\) be a prime and suppose that \(P\) is a \(p\)-group with no non-cyclic abelian subgroups. Then either \(P\) is cyclic or \(p=2\) and \(P\)  is generalised quaternion.
\end{lemma}
\begin{proof}
     See \cite[Theorem 4.10 (ii)]{MR0569209}.
\end{proof}
 Since \(G\) is soluble, every minimal normal subgroup  \(N\) is elementary abelian. That is \(N\)  is a direct product of cyclic groups of order $p$ for some prime \(p\). Therefore we can regard \(N\) as a vector space over the field of order $p$, $\mathrm{GF}(p)$.  The following lemma is particularly useful when we view   minimal normal subgroups   in this way.

\begin{lemma}\label{Aschbacher}
    Let \(p,q\) and \(r\) be distinct primes, \(X\) a group of order \(r\) faithful on a \(q\)-group \(Q\), and \(V\) a faithful \(\mathrm{GF}(p)XQ\)-module. If \(q=2\) and \(r\) is a Fermat prime assume \(Q\) is abelian. Then \(C_{V}(X)\neq 0\).
\end{lemma}
\begin{proof}
     See \cite[Theorem 36.2]{MR1777008}.
\end{proof}
\begin{lemma}\label{cyclic subgroups}
    Let \(N\) be an elementary abelian \(p\)-group and let \(Q\) be a non-cyclic abelian \(q\)-subgroup of \(\mathrm{Aut}(N)\), where \(q\) and \(p\) are distinct primes. Then \(N=\Pi_{a\in Q^{\#}}C_{N}(a)\).
\end{lemma}
\begin{proof}
     See \cite[Theorem 3.3]{MR0569209}
\end{proof}

\begin{lemma}\label{Thompson}
    If \(G\) admits a fixed-point-free automorphism of prime order, then \(G\) is nilpotent.
\end{lemma}
\begin{proof}
    See \cite[Theorem 10.2.1]{MR0569209}.
\end{proof}

\section{Frobenius groups and disconnected normalising graphs}

We recall some properties of Frobenius groups.
\begin{lemma}\label{Frobenius}
    If \(G\) is a Frobenius group with kernel \(K\) and complement \(C\), then the following conditions hold:
    \begin{enumerate}[label=(\roman*)]
        \item \(C\) induces a regular group of automorphisms of \(K\).
        \item \(|C|\) divides \(|K|-1\).
        \item \(K\) is nilpotent and is abelian if \(|C|\) is even.
        \item The Sylow \(p\)-subgroups of \(C\) are cyclic for odd \(p\) and cyclic or generalised quaternion if \(p=2\).
        \item Any subgroup of \(C\) of order \(pq\), where \(p\) and \(q\) are primes, is cyclic.
        \item If \(|C|\) is odd, then \(C\) is metacyclic. If \(|C|\) is even, then \(C\) possesses a unique involution and this is contained in \(Z(C)\).
    \end{enumerate}
\end{lemma}
\begin{proof}
    See \cite[Theorem 10.3.1]{MR0569209}.
\end{proof}

\begin{lemma}\label{odd sol frobenius}
     If \(C\) is a soluble Frobenius complement of odd order and \(s \in C\) has prime order, then $\langle s\rangle$ is a normal subgroup of $C$.
\end{lemma}
\begin{proof}
     Let \(s\in C\) be such that \(o(s)=t\) for some \(t\in \pi(C)\) and assume that $\langle s\rangle$ is not normal in $C$. Since \(C\) is a Frobenius complement of odd order,   Lemma \ref{Frobenius} (iv)  implies that the Sylow $t$-subgroups of $C$ are cyclic.

     If \(s\in  F(C)\),  then \(\langle s\rangle=\Omega_{1}(O_{t}(C))\). Since \(\Omega_{1}(O_{t}(C)) \) is characteristic in  \(O_{t}(C)\) and \(O_{t}(C)\) is normal in \(C\), we  have \(\langle s\rangle\) is normal in  \(C\), which is a contradiction. Therefore $s \not \in F(C)$. Hence, as $C_C(F(C))\le F(C)$, there exists \(r\in \pi(C)\) such that \(O_{r}(C)\) is not centralised by $s$. As $r$ is odd, $O_r(C)$ is cyclic by Lemma~\ref{Frobenius} (iv). Thus  Lemma~\ref{Frobenius} (v) implies that
      $\Omega_1(O_r(C))\langle s\rangle$ is cyclic. In particular $s$ centralises $\Omega_1(O_r(C))$ and this means that $s$ centralises $O_r(C)$  by \cite[(24.9)]{MR1777008}, which is a contradiction. This proves the lemma.
\end{proof}

We  continue  this section with an elementary lemma about the connected components of normalising graphs of Frobenius groups.

\begin{lemma}\label{connected components}
    Suppose that $G=KC$ is a Frobenius group and that $\Gamma(G)$ is disconnected.  Then the connected components of $\Gamma(G)$ are $\Gamma(K)$ and $\Gamma(C^k)$ for $k \in K$.  Furthermore, the connected components of $\Gamma(G)$ have diameter at most $2$.
\end{lemma}

\begin{proof}
 As $C$ has a normal cyclic subgroup by Lemmas~\ref{Frobenius} (vi) and \ref{odd sol frobenius} and $K$ is nilpotent, both the $K$ and $C^k$, $k\in K$  have connected normalising graph of diameter at most $2$. If $a \in C^k$ and $b \in C$  with $a$ normalising $\langle b\rangle$, then $b^a \in C \cap C^a$ which means that $C= C^a$ and $a \in C\cap C^k$. Thus $C=C^k$.  Hence there are no edges between distinct complements in $G$.  Since $\Gamma(C)$ and $\Gamma(K)$ are connected, there are also no edges between $\Gamma(C)$ and $\Gamma(K)$. This proves the main claim.
\end{proof}

Notice that the proof of Lemma~\ref{connected components} shows that in  the normalising graph of a Frobenius group, there are no edges between elements of different Frobenius complements.

\begin{proposition}\label{disconnected}
    Let \(G=KC\) be a Frobenius group with kernel \(K\) and complement \(C\). Then \(\Gamma(G)\) is disconnected if and only if for all \(p\in\pi(C)\) and \(r\in\pi(K)\) we have \(p \nmid r-1\). Furthermore, if \(\Gamma(G)\) is connected, then \(\mathrm{diam}(\Gamma(G))\leq 4\).
\end{proposition}

\begin{proof}   Observe that, as \(K\) is nilpotent, \(\Gamma(K)\) is connected and Lemmas \ref{Frobenius} (vi) and  \ref{odd sol frobenius} imply that \(\Gamma(C^g)\) is connected for all $g \in G$.

Since $G =K \cup \bigcup_{g \in G}C^g$, it follows that $G$ is connected if and only if an element of $G\setminus K=(\bigcup_{g \in G}C^g)^\#$  forms an edge with some element of $K$. As $K$ is normal in $G$, this happens   if and only if there is an edge $c\sim k$ in $\Gamma(G)$ with $c \in C$ and $k \in K$.

Suppose that $G$ is connected and that $c \sim k$ is an edge in $\Gamma(G)$ with $c \in C$ and $k\in K$. We may assume that  $o(c)=p$ and $o(k)=r$ for some $p\in \pi(C)$ and $r \in \pi(K)$.  By Lemma~\ref{Frobenius} (ii), $p\ne r$. Since \(G\) is a Frobenius group, \(C^{g}\cap C=1\) for all \(g\in G \setminus C\). Thus  \(c\in N_{G}(\langle k \rangle)\). Furthermore, as \(C_{G}(k)\leq K\) for all \(k\in K^{\#}\), we  have \(c\in N_{G}(\langle k \rangle)\setminus C_{G}(\langle k \rangle)\).  Since \(|\mathrm{Aut}(\langle k \rangle)|= r-1\), we necessarily have \(p \mid (r-1)\).

Conversely, assume that \(p\mid(r-1)\) for some \(p\in \pi(C)\) and \(r\in \pi(K)\). We intend to show that $\Gamma(G)$ is connected. Aiming for a contradiction, assume that $\Gamma(G)$ is disconnected.  Let $V \le Z(O_p(K)) \le Z(K)\le K$ be a minimal normal subgroup of $K\gen{c}$. Since $\Gamma(G)$ is disconnected, $V$ is not cyclic and so $|V| = r^m$ with $m>1$. The minimality of $V$ implies that $\gen{c}$ does not normalise a non-trivial cyclic subgroup of $V$. Let $\mathcal C$ be the set of non-trivial cyclic subgroups of $V$. Then, as  \(p\mid (r-1)\), we  have \(r\equiv 1 \pmod p\) and so
$$|\mathcal C|= \frac{r^{m}-1}{r-1} = r^{m-1}+\ldots+ r+ 1 \equiv m\pmod p.$$
As $|\mathcal C|\equiv 0 \pmod p$, we conclude that
$m=p$.  Therefore, for $v \in V^\#$, we have $V=\gen{v,v^{c},v^{c^{2}},\ldots v^{c^{p-1}}}$ has order $r^p$. However, in this case,  $w=vv^{c}v^{c^{2}}\ldots v^{c^{p-1}}\in V$ is non-trivial and centralised by $\gen{c}$. Thus $c\sim w$ and we have a contradiction.  Hence $\Gamma(G)$ is connected.

Finally we show that if \(\Gamma(G)\) is connected, then \(\mathrm{diam}(\Gamma(G))\leq 4\). Observe that the elements which lie in distinct conjugates of \(C\) are the greatest distance apart of all pairs of elements in \(G^{\#}\). Therefore it suffices to show that the distance between any two such elements is at most \(4\).  Let \(x \in C\) and  $x^*\in C^{g}$ for some \(g\in G\).  Since \(\Gamma(G)\) is connected there exists \(y\in C\) such that \(o(y)= r\) where \(r\in\pi(C)\) and \(r\mid (p-1)\) for some \(p\in\pi(K)\). Thus \(y\sim n\) where \(n\in Z(K)\) with \(o(n)=p\). By Lemma \ref{Frobenius} (iv) and  \ref{odd sol frobenius} we know that \(\langle y\rangle \unlhd C\). Hence, \(y\) shares an edge with every element of \(C^\#\).
Since \(G=CK\) we may write \(g=ck\) for some \(c\in C\) and  \(k\in K\).  Then \(C^{g}=C^{ck}=C^{k}\) and in particular \(\langle y^{k}\rangle \unlhd C^{g}\). Furthermore, \(y^{k}\sim n^{k}\). However, as $n\in Z(K)$,   \(n^{k}=n\) and we obtain the path
\[ x \sim y \sim n\sim y^{k} \sim x'.
\]
Thus, if \(G\) is Frobenius group and \(\Gamma(G)\) is connected, then \(\mathrm{diam}(\Gamma(G))\leq 4\).
\end{proof}

\section{The distance from a minimal normal subgroup }

 From here on $G$ is a finite soluble group and $\Gamma(G)$ is its normalising graph.
The goal of this section is to establish the following theorem.

\begin{theorem}\label{norm distance}
    Suppose that \(G\) is a finite soluble group, $N$ is a minimal normal subgroup of $G$ and $x \in G^{\#}$. Then either $\mathrm{d}(x,N) \le 3$  or   \(G\) is a Frobenius group.
\end{theorem}

 Assume that $N$ is a minimal normal subgroup of $G$. Then $N$ is an elementary abelian $p$-group for some prime $p$.
Our proof of Theorem \ref{norm distance} begins by considering an element \(x\in G\setminus N\) such that \(\mathrm{d}(x,N)>3\). This element is fixed for the duration of this section.  The objective is to prove that in this case, $G$ is a Frobenius group and $\Gamma(G)$ is disconnected.  We investigate the properties that   \(x\) must possess.  This allows us to determine the global structure of \(G\). We fix $y \in \langle x \rangle $ of prime order $r$ and note that $x\sim y$ or $x=y$.  Abusing notation we  always write $x\sim y$.    We begin with the following lemma which does not in its first conclusion require $\mathrm{d}(t,N) >3$.

\begin{lemma}\label{lem:odd order} Suppose that $t \in G$ is an involution.  Then $\mathrm{d}(t,N) \le 1$. In particular, $C_G(y)$ and $x$ have odd order.
\end{lemma}
\begin{proof} Let \(n\in N\) be non-trivial.  We may assume $t \not \in N$ and that $t$ does not commute with $n$.
    Then \(\langle t,t^{n}\rangle\) is  a dihedral group.  Furthermore, \(\langle tt^{n} \rangle \)  is normal in \(\langle t,t^{n}\rangle\) and $$tt^{n}=tn^{-1}tn=(n^{-1})^{t}n\in N.$$ So  \( t \sim tt^{n}\in N\) is an edge . It follows that $\mathrm{d}(t,N) \le 1$.

Suppose that $t$ can be chosen in $C_N(y)$.   Then $x\sim y\sim t \sim N$ exhibits a path of length at most $3$ to $N$,  which is a contradiction as $\mathrm d(x,N)>3$. Hence $C_G(y)$ and $x$ have odd order.
\end{proof}
\begin{lemma}\label{p element}
    Let \(z\) be an element of order a power of \(p\). Then \(\mathrm{d}(z,N)\leq 1\).  In particular, \(o(x)\) is coprime to \(p\).
\end{lemma}
\begin{proof}
    Since \(z\) has order a power of $p$, \(z\in P\) for some \(P\in \mathrm{Syl}_{p}(G)\). Furthermore, \(N\unlhd P\). Since \(N\cap Z(P)\neq 1\), we obtain the path \(z \sim N\).
     Suppose that \(p$  divides $o(x)\). Then there exists \(w\in \langle x \rangle \) such that \(o(w)=p\). We then discover the path \(x \sim w \sim N\) which is impossible. Thus \(o(x)\) is coprime to \(p\).
\end{proof}

Next, we determine some properties of the subgroup structure of \(G\) first beginning with  the embedding of $F=F(G)$ in $G$.

\begin{lemma}\label{lem:Fit1} We have  \(F=C_{G}(N)\) and \(C_{G}(  y )\cap F=1\).  In particular, $o(y)$ and $|F|$ are coprime.
\end{lemma}
\begin{proof} Suppose that $y$ centralises a non-trivial element $z$  of $C_G(N)$.  Then we have a path \(x\sim y\sim z \sim N\), contradicting our choice of \(x\). Hence $y$ acts fixed-point-freely on $C_G(N)$ and Lemma \ref{Thompson} yields $C_G(N)$ is nilpotent and $o(y)$ and $|F|$ are coprime.  Hence $N\le C_G(N)\le F$ and $C_G(y) \cap C_G(N)=1$.
      As \(F\) is nilpotent and \(N\unlhd F\) we know that \(1\neq N\cap Z(F) \unlhd G\). However, \(N\)  is a minimal normal subgroup  of  \(G\) and \(N\cap Z(F) \) is normal in $G$. Thus \(N\leq Z(F)\) and \(F\leq C_{G}(N)\). Hence \(F=C_{G}(N)\). This concludes the proof.
\end{proof}

\begin{lemma}\label{lem:Frob1}
       We have \(C_G(y)N\) is a Frobenius group with complement \(C_{G}(y)\) of odd order.
\end{lemma}

\begin{proof}
Let $w \in C_G(y)N\setminus C_G(y)$.  Then $w= cn$ for some $c \in C_G(y)$ and non-trivial $n \in N$. We have $C_G(y)^w= C_G(y)^{cn}=C_G(y)^n$.

Assume that $d \in C_G(y)^w \cap C_G(y) =C_G(y)^n \cap C_G(y)$ is non-trivial.  Then $d^{n^{-1}},d\in C_G(y)$. This means that $$d^{-1}ndn^{-1} =n^dn^{-1}\in  C_G(y)\cap N \le C_G(y)\cap F =1$$ by Lemma~\ref{lem:Fit1}. Therefore $d$ centralises $n$ and  $$x\sim y \sim d \sim N,$$ a contradiction. Hence $C_G(y)^w\cap C_G(y)=1$ for all $w \in C_G(y)N\setminus C_G(y)$ and so $C_G(y)N$ is a Frobenius group with Frobenius complement $C_G(y)$. Furthermore, $|C_G(y)|$ is odd by Lemma~\ref{lem:odd order}.
\end{proof}

 Recall that $y$ has prime order $r$.

\begin{lemma}\label{cyclic sylows}
    The  Sylow \(r\)-subgroups of \(G\) are cyclic.
\end{lemma}
\begin{proof}
    Let  \(R\in \mathrm{Syl}_{r}(G)\) with $y \in R$.  If \(R\) is not cyclic we can find an element \(z\in R\) of order \(r\) such that \(\langle z,y\rangle\) is elementary abelian by Lemma~\ref{no non cyclic subgroup in P group}.   This contradicts the  combination of Lemmas~\ref{Frobenius} (iv)   and \ref{lem:Frob1}. 
\end{proof}
\begin{lemma}\label{yF normal}
     \(\langle y\rangle F \unlhd G\)
\end{lemma}
\begin{proof}
     Let \(H=F(G/F)\) and  \(\overline{G}=G/F\). Then $$H=\left(\Pi_{t\in\pi(H)\setminus{\{r\}}}O_{t}(\overline{G})\right)\times O_{r}(\overline{G}).$$  By Lemma~\ref{cyclic sylows}, the Sylow \(r\)-subgroups of \(G\) are cyclic so this is also the case for \(\overline{G}\). Hence it suffices to show that \(O_{r}(\overline{G})\ne 1\). Suppose for a contradiction that \(O_{r}(\overline{G})=1\). Then there exists \(t\in \pi(H)\) such that \(\overline{\langle y\rangle }\) does not centralise \(O_{t}(\overline{G})\). We know \(O_{t}(\overline{G})\langle yF\rangle\) acts on \(N\).

     Since $C_N(y)=1$ by Lemma~\ref{lem:Fit1}, Lemma \ref{Aschbacher} implies that $t=2$ and \(yF\) centralises $Z(O_{2}(\overline{G}))$.
     Let $ Z>F $ be such that $ \overline Z= Z(O_{2}(\overline{G}))$. Then $Z$ is normalised by $y$ and, as $r\ne t$,  Lemma~\ref{lem:Fit1}  shows $|Z|$ and $o(y)$ are coprime. Since $yF$ centralises $\overline Z$, coprime action yields $C_G(y)$ has even order and this contradicts Lemma~\ref{lem:odd order}.
        Hence  \(\langle y\rangle F \unlhd G\), as claimed.
\end{proof}

We are now in a position to prove Theorem \ref{norm distance}.
\begin{proof}[Proof of  Theorem \ref{norm distance}.]
    We continue to suppose that \(x\in G\setminus N\) is such that \(\mathrm{d}(x,N)>3\)  and $y \in \langle x\rangle$ has prime order.  We  show that \(N_{G}(\langle y\rangle)\) is a Frobenius complement to \(F\).  If $w \in N_F(\langle y\rangle)$, then $[w,y]\in F \cap \langle y\rangle =1$. Thus $N_G(\langle y\rangle) \cap F= C_G(y) \cap F=1$ by Lemma~\ref{lem:Fit1}.  As $\langle y \rangle$ is a Sylow $r$-subgroup of \(\langle y\rangle F\),    the Frattini argument shows that \(G=N_{G}(\langle y\rangle)F\).  In particular, $N_G(\langle y\rangle)$ is a complement to $F$ in $G$. We now show that it is a Frobenius complement.

    Let $g \in G\setminus N_G(\langle y \rangle)$ and write $g=wf$ with $w \in N_G(\langle y \rangle)$ and $f$ a non-trivial element in $ F$. Then $ N_G(\langle y \rangle)^g= N_G(\langle y \rangle)^f$  and $ N_G(\langle y \rangle)^g\cap  N_G(\langle y \rangle)=  N_G(\langle y \rangle)^f \cap  N_G(\langle y \rangle)$. Assume that $h $ is a non-trivial element of prime order in this intersection.  Then $h^{f^{-1}}, h \in  N_G(\langle y \rangle)$. Consequently $fhf^{-1}h^{-1} \in N_G(\langle y \rangle)\cap F=1$. Hence $f$ commutes with $h$. Furthermore, both $f$ and $h$ are non-trivial.  As $N \le Z(F)$,  we now have a path $$x\sim y \sim h\sim f\sim N.$$  If $h\not \in C_G(y)$, then $\langle h\rangle \langle y\rangle$ acts on $N$. Since the Sylow $r$-subgroups of $G$ are cyclic by Lemma~\ref{cyclic sylows},  the order of  $h$ is coprime to $r$.  Hence Lemma~\ref{Aschbacher} yields a path $x\sim y \sim h \sim N$, which is impossible.  Thus $h$ centralises $y$. Since $C_G(y)$ is an odd order Frobenius complement by Lemma ~\ref{lem:Frob1}, we can apply Lemma \ref{odd sol frobenius} to see that \(\langle h\rangle\) is normal in \(C_{G}(\langle y\rangle)\). Hence the path above can be shortened to \(x \sim h \sim f \sim N\). Once again this contradicts our choice of \(x\).  Therefore $ N_G(\langle y \rangle)^g\cap  N_G(\langle y \rangle)=1$ for all $g \in G\setminus N_G(\langle y\rangle)$ and so $G$ is a Frobenius group with Frobenius complement $N_G(\langle y\rangle)$ as claimed.
     This completes the proof of the theorem.

\end{proof}

\section{Bounding the diameter of the normalising graph }

The objective of this section is to establish parts (i) and (ii) of Theorem~\ref{Diam 6}. The main struggle is to show that when $G$ is not a Frobenius group, then the diameter of $\Gamma(G)$ is not $7$.

\begin{proof}[Proof of Theorem \ref{Diam 6} (i) and (ii)] Suppose that $G$ is not a Frobenius group and that the diameter of $\Gamma=\Gamma(G)$ is at least $7$. We seek a contradiction. Set $F=F(G)$ and  $\pi=\pi(F)$.  Then $G>F$ for otherwise $G$ has diameter at most $2$.  Let $N$ be a minimal normal subgroup of $G$.
Then $\mathrm{d}(x,N)\le 3$ for all $x \in G$ by Theorem~\ref{norm distance} and $\Gamma$ is connected.   Assume that    \(\mathrm{d}(x,x^{*})\ge 7\)  for some $x, x^*\in G$.
Then there exist $a,a^*,b,b^* \in G^\#$ of prime order and  paths $$x\sim a\sim b \sim n$$ and $$x^*\sim a^*\sim b^*\sim n^*$$ with $n, n^* \in N$ and $n \ne n^*$ for otherwise $\mathrm d(x,x^*)\le 6$. Because $N$ is abelian, $n\sim n^*$ and so $\mathrm{d}(x,x^{*})= 7$ and
  $$x\sim a\sim b \sim n\sim n^*\sim b^*\sim a^*\sim x^*.$$

  If $b$ can be chosen in $C_G(N)$, then $b$ is adjacent to all elements of $N$ and   we truncate the path to
  $$x\sim a\sim b \sim n^*\sim b^*\sim a^*\sim x^*$$ to obtain the contradiction
  $\mathrm d(x,x^*)\le 6$. Hence we must have \begin{equation}\label{clm1} \mathrm d(x,C_G(N))=3 =\mathrm d( x^*,C_G(N)).\end{equation}
In particular, for all $y \in \langle x\rangle$ of prime order,\begin{equation}\label{clm2} C_G(y) \cap C_G(N)=C_N(y)=1.\end{equation}

By (\ref{clm2}), all elements of  $\langle x \rangle$ of prime order act fixed-point-freely on $C_G(N)$
  and so Lemma~\ref{Thompson} implies that $C_G(N)$ is nilpotent. As $N \le Z(F)$, we conclude that \begin{equation}\label{clm3}F =C_G(N).\end{equation}

 If $r\in \pi$ divides $o(x)$, then we may choose $y\in \langle x\rangle$ of order $r$ such that, using  (\ref{clm3}), $$1\ne C_{O_r(F)}(y)\le C_F(y)= C_G(y) \cap F = C_G(y) \cap C_G(N)$$ and this contradicts (\ref{clm2}).  Hence $x$, and similarly $x^*$, have order coprime to  $|F|$.
\begin{equation}\label{clm3.5} \text{The elements }x \text{ and }   x^*\text{ have order coprime to  }|F|.\end{equation}

   Because  $G>F$,  there exists $t \in \pi(G)$   such that $O_t(G/F) \ne 1$. Let $K$ be the preimage of $O_t(G/F)$, $T \in \mathrm{Syl}_t(K)$, and $H=N_G(T)$. As $K$ is normal in $G$, the Frattini argument implies $G= KH= FTH= FH$.

      Since $G=FH$, $H$
 contains a Hall $\pi'$-subgroup of $G$.    Therefore we may assume that $$x \in H=N_G(T)\text{ and } x^*\in H^g=N_G(T^g)$$  for some $g \in F=C_G(N)$.

We claim
\begin{equation}\label{clm4} t \not \in \pi.\end{equation}

    If $t \in \pi$, then $T \cap F \ne 1$ and we may choose  a minimal normal subgroup  $M$ of $G$ with $M \le O_t(F) \le T$.  Hence $C_M(T)=Z(T) \cap M \ne 1$ and, as $TF$ is normal in $G$, we deduce   that $F<TF\le C_G(M)$.  However,   $M$ is a candidate for our initial choice of $N$, and so (\ref{clm3}) applies to $M$, and this yields the contradiction $F<C_G(M)=F$   and  proves (\ref{clm4}).

\medskip

As $b \not \in C_G(N)= F$ by (\ref{clm1}) and (\ref{clm3}), $b \sim n$ implies $b\in N_G(\langle n \rangle)\setminus  F$. Hence  $N_G(\langle n \rangle)> F$ and
 \[N_{G}(\langle n\rangle) = N_{G}(\langle n\rangle)\cap G = N_{G}(\langle n\rangle)\cap FH = F(N_{G}(\langle n\rangle)\cap H). \]
   This implies that \(N_{G}(\langle n\rangle)\cap H \neq 1\). A similar argument shows that \(N_{G}(\langle n\rangle)\cap H^{g} \neq 1\).
Hence we may choose non-trivial elements
$w \in  N_{G}(\langle n\rangle)\cap H$ and $w^* \in  N_{G}(\langle n\rangle)\cap H^g$.
Thus $$w\sim n \sim w^*.$$

Suppose that   $Z(T)=\langle u\rangle$ is cyclic. Then  $x$ and $w$ normalise $\langle u \rangle$  and $x^*$ and $w^*$ normalise $\langle u^g\rangle$.
Therefore we obtain a modified path $$x\sim u \sim w\sim n \sim w^* \sim u^g \sim x^*,$$ which is impossible as $\mathrm d(x,x^*)=7$. We conclude that
\begin{equation}\label{clm5}Z(T)\text{ is not cyclic}.\end{equation}

\medskip

Assume that $x$ is a $t$-element and $y\in \langle x\rangle$ has order $t$. If $t=2$, then $y\sim N$ by Lemma~\ref{lem:odd order}, contrary  to $\mathrm{d}(x,N)= 3$ by (\ref{clm1}). Hence $t$ is odd.  As $x \in H= N_G(T)$, $T\langle x\rangle$ is a $t$-group.  Set $U=  Z(T)\langle x\rangle$.  If \(N_{U}(\langle x \rangle)\) is not cyclic, there exist \(\ell\in N_{U}(\langle x \rangle)\) such that \(\langle y,\ell\rangle\) is elementary abelian of order $t^2$ by Lemma~\ref{no non cyclic subgroup in P group}.  By Lemma \ref{cyclic subgroups}, we can find some non-trivial \(s\in \langle y,\ell\rangle \le  N_{U}(\langle x \rangle)\) such that \(C_{N}(s)\neq 1\). Because $\langle x\rangle$ is normalised by $\langle s\rangle$, this results in the path \(x\sim s \sim N\) contrary to (\ref{clm1}).  Consequently  \( N_{U}(\langle x \rangle )\) is cyclic.
However, this means that \(N_{U}(N_{U}(\langle x \rangle))=N_{U}(\langle x \rangle)\). Therefore \(\langle x \rangle \unlhd
     U=N_U(\langle x\rangle)$ which is cyclic.  But $Z(T) \le U$ is not cyclic by (\ref{clm5}), and so this is impossible. Thus $x$ is not a $t$-element.  A similar argument works to show that $x^*$ is not a $t$-element. Hence
     \begin{equation}\label{clm6}x \text{ and } x^* \text{ are not }t\text{-elements.}\end{equation}

     By (\ref{clm6})  there exists \(r\in\pi(\langle x \rangle)\setminus\{t\}\) and  \(r^*\in\pi(\langle x^* \rangle)\setminus\{t\}\).
      Let $y \in \langle x \rangle$ have order $r$  and $y^* \in \langle x^* \rangle$ have order $r^*$. If $[Z(T),y]\ne 1$, then, as $C_G(N)=F$, $\{r,t\} \cap \pi=\emptyset$ by (\ref{clm3.5}) and (\ref{clm4}) and $Z(T)\langle y \rangle \cap F=1$,  Lemma \ref{Aschbacher} implies that $C_N(y)\ne 1$, a contradiction to (\ref{clm2}).  Thus $Z(T)$ and $y$ commute. Similarly,  $[Z(T^g),y^*]=1$.

     \begin{equation}\label{clm7} Z(T) \le C_G(y)\text { and } Z(T^g) \le C_G(y^*).\end{equation}

     Since $Z(T)$ is not cyclic by (\ref{clm5}), Lemma \ref{cyclic subgroups} implies that there exists $u\in Z(T)^\#$ such that $C_N(u)\ne 1$. Let $m \in C_N(u)^\#$.  Recall now that $g\in F=C_G(N)$ and so $g$ centralises $N$ and, in particular, $m= m^g\in C_N(u^g)$.

       As $u$ commutes with $y$ and $u^g$ commutes with $y^*$  by (\ref{clm7}), we have $$x\sim y \sim u \sim m= m^g\sim u^g\sim y^*\sim x^*$$ which again yields the contradiction $\mathrm d (x,x^*)\le 6$. This is our final contradiction.
        We conclude that \(\mathrm{diam}(\Gamma ) \leq 6\).

       We have shown that $G$ is either a Frobenius group or $\Gamma$ has diameter at most $6$. In the case that $G$ is a Frobenius group, Lemma~\ref{connected components} and Proposition~\ref{disconnected} provide the additional information given in Theorem~\ref{Diam 6} (i). This completes the proof of parts (i) and (ii) of   Theorem~\ref{Diam 6}.
\end{proof}

\section{A soluble group with permuting graph of diameter $6$}
In this section we present an example of a finite soluble group whose permuting graph has diameter \(6\). We briefly outline our approach.

Let \(N\) be the \(6\)-dimensional vector space over the field of \(5\) elements and define \(H \le \mathrm{GL}(N)\) to be the group  $\langle t_1, t_2, x\rangle$ where
$$t_1=\left(\begin{smallmatrix}
    1 & 0 & 0 & 0 & 0 & 0\\
    0 & 1 & 0 & 0 & 0 & 0\\
    0 & 0 & -1 & 0 & 0 & 0\\
    0 & 0 & 0 & -1 & 0 & 0\\
    0 & 0 & 0 & 0 & -1 & 0\\
    0 & 0 & 0 & 0 & 0&-1
\end{smallmatrix}\right)\; t_2=
\left(\begin{smallmatrix}
    -1 & 0 & 0 & 0 & 0 & 0\\
    0 & -1 & 0 & 0 & 0 & 0\\
    0 & 0 & -1 & 0 & 0 & 0\\
    0 & 0 & 0 & -1 & 0 & 0\\
    0 & 0 & 0 & 0 & 1 & 0\\
    0 & 0 & 0 & 0 & 0 & 1
\end{smallmatrix}\right)\text{  and } x=
\left(\begin{smallmatrix}
    0 & 0 & -1 & 1 & 0 & 0\\
    0 & 0 & -1 & 0 & 0 & 0\\
    0 & 0 & 0 & 0 & 1 & 0\\
    0 & 0 & 0 & 0 & 0 & 1\\
    1 & 0 & 0 & 0 & 0 & 0\\
    0 & 1 & 0 & 0 & 0 & 0
    \end{smallmatrix}\right).$$
  Then $t_1$ and $t_2$ are involutions, $x$ has order $9$ and $H$ has order $36=2^2.3^2$. Set \(G=N\rtimes H\) which has order $5^6.3^2.2^2$.
We calculate that \(x\) acts  fixed-point-freely on \(N\). Furthermore, \(N_{H}(\langle x\rangle )\) does not contain any of the involutions in \(H\) (or any of the elements of order \(6\)). Power $x$ to get that \[x^{3}=\left(\begin{smallmatrix}
     -1 & 1 & 0 & 0 & 0 & 0\\
    -1 & 0 & 0 & 0 & 0 & 0\\
    0 & 0 & -1 & 1 & 0 & 0\\
    0 & 0 & -1 & 0 & 0 & 0\\
    0 & 0 & 0 & 0 & -1 & 1\\
    0 & 0 & 0 & 0 & -1&0
 \end{smallmatrix}\right)
 \] and note that \(x^{3}\) acts fixed-point-freely on \(N\). Thus, as $3$ does not divide $4=5-1$, there is no edge in the permuting graph between any non-trivial element of \(\langle x\rangle\) and any non-trivial element of \(N\). Since $\langle x\rangle$ only permutes with $\langle x^3\rangle$, the neighbours of $x$ have order $3$ or $9$ and they are contained in $H$.  Now the elements of order $3 $ in $H$ are joined to the elements of order $9$ in $H$ and to the non-trivial elements in $C_H(x)\setminus\{x,1\}$.  So at distance two from $x$ in $\Psi(G)$ we only see elements of $H$.
Now let  $w$ be the product of the standard basis elements of $N$.     We   check that no non-trivial element of $H$ is joined to $H^w$ in the permuting graph using {\sc Magma} \cite{magma}. This shows that the diameter of $\Psi(G)$ is either infinite or $6$. Since $G$ is not a Frobenius group,  Theorem \ref{Diam 6} (i) and (ii) imply that \(\mathrm{diam}(\Gamma(G))=\mathrm{diam}(\Psi(G))=6\). This example concludes the proof of Theorem~\ref{Diam 6} and Corollary~\ref{cor:permuting graphs}.

\bibliographystyle{abbrv}
\bibliography{PermGraph}
\end{document}